\documentclass[a4paper]{amsart}

\usepackage{amssymb}
\usepackage{amsmath,amsthm}
\usepackage{enumerate,paralist}
\usepackage{amstext}
\usepackage{psfrag}
\usepackage{dsfont}
\usepackage[dvips]{epsfig}
\usepackage[english]{babel}
\usepackage{color}
\usepackage{mathrsfs}
\usepackage{footnote}

\theoremstyle{plain}
\newtheorem{theorem}{Theorem}
\newtheorem{corollary}{Corollary}
\newtheorem{lemma}{Lemma}
\newtheorem{proposition}{Proposition}
\theoremstyle{definition}

\newtheorem{example}{Example}
\newtheorem{remark}{Remark}

\numberwithin{theorem}{section}
\numberwithin{corollary}{section}
\numberwithin{lemma}{section}
\numberwithin{definition}{section}
\numberwithin{example}{section}
\numberwithin{remark}{section}
\numberwithin{proposition}{section}

\renewcommand{\leq}{\leqslant}
\renewcommand{\geq}{\geqslant}

\newcommand{\1}{\mathds 1}

\def\<{\left\langle}
\def\>{\right\rangle}

\newcommand{\dd}[1]{\bar{#1}}

\def \ff{\xi}
\def \pN{\mathcal{N}}
\def \NN{\mathbb{N}}
\def \sp{p^*}
\def \tp{\tilde{p}}
\def \tP{\tilde{P}}
\def \Ps{P^*}
\def \RR{\mathbb{R}}
\def \Rd{{\RR^d}}

\definecolor{kb}{rgb}{0.1,0.5,0.1}
\definecolor{yk}{rgb}{0.1,0.2,0.7}
\definecolor{ks}{rgb}{0.7,0.1,0.2}

\begin{document}

\title[Majorization and 4G Theorem]{
Majorization, 4G Theorem
and Schr\"odinger perturbations
}



\thanks{Krzysztof Bogdan was partially supported by NCN grant 2012/07/B/ST1/03356. Karol Szczypkowski was partially supported by NCN grant  2011/03/N/ST1/00607. } 

\author{Krzysztof Bogdan}
 \address{Wroc{\l}aw University of Technology,
Wybrze{\.z}e Wyspia{\'n}skiego 27,
50-370 Wroc{\l}aw, Poland}
\email{bogdan@pwr.edu.pl}

\author{Yana Butko}
\address{Bauman Moscow State Technical University, 105005, 2nd Baumanskaya str. 5, Moscow, Russia and   University of Saarland,  P.O. Box 15 11 50, D-66041 Saarbr\"{u}cken, Germany}
\email{yanabutko@yandex.ru, kinderknecht@math.uni-sb.de}

\author{Karol Szczypkowski}
\address{Universit{\"a}t Bielefeld, Postfach 10 01 31,
D-33501 Bielefeld, Germany  and
Wroc{\l}aw University of Technology,
Wybrze{\.z}e Wyspia{\'n}skiego 27,
50-370 Wroc{\l}aw, Poland
}
\email{karol.szczypkowski@math.uni-bielefeld.de, karol.szczypkowski@pwr.edu.pl}

\date{\today}

\maketitle

\begin{abstract}
Schr\"odinger perturbations of transition densities by singular potentials may fail to be comparable with the original transition density. For instance this is so for 
the transition density of a subordinator perturbed by any time-independent unbounded potential.
In order to estimate such perturbations it is convenient to use an auxilary transition density as a majorant 
and
the 4G inequality
for the original transition density and the majorant.
We prove the 4G inequality for the $1/2$-stable and inverse Gaussian subordinators, discuss the corresponding class of admissible potentials and  indicate estimates for the resulting transition densities of Schr\"odinger operators. 
The connection of the transition densities to their generators is made via the weak-type notion of fundamental solution.

\end{abstract}

\maketitle

\section{Introduction and Preliminaries}\label{sec:intro}

Schr\"odinger perturbation consists of adding to a given operator an operator of multiplication by a function $q$.
On the level of inverse operators the addition results in the perturbation series. We focus on transition densities $p$ perturbed by functions $q\ge 0$.
Our main goal is to give pointwise estimates for the resulting perturbation series $\tp$ under suitable integral conditions on $p$ and $q$. For instance, bounded potentials $q$ produce transition densities $\tp$ comparable with the original $p$ in finite time.
In a series of recent papers,  integral conditions leading to  comparability of $\tp$ and $p$  were proposed which allow for rather singular potentials $q$, if $p$ satisfies the 3G Theorem \cite{MR2457489, MR3000465}.
The integral conditions compare the second term in the perturbation series
(that which is linear in $q$) with $p$ (the first term of the series). The comparison is meant to prevent the instantaneous blowup and to control the long-time accumulation of mass. The first property 
may be secured by smallness conditions, like $0\le \eta<1$ below, and the second is accomplished by using a subadditive function $Q$. The results render $p$ an approximate majorant for $\tp$ in finite time \cite{MR3000465}. They may also be considered as analogues of the Gronwall inequality \cite{MR3065313}.
We note that similar estimates for Green-type kernels were recently obtained in \cite{Frazier:2014qsa}, \cite{MR2207878}, \cite{MR2403428}.

The 3G Theorem, which is related to the quasi-metric condition \cite{Frazier:2014qsa}, is common for transition densities with power-type decay, e.g., the transition density of the fractional Laplacian.
However, many transition densities fail to satisfy 3G, for instance
the Gaussian kernel.
In \cite{MR3200161} and \cite{MR3065313} a more flexible majorization technique is proposed, motivated by earlier results of \cite{MR1488344}. Namely, another transition density $p^*$ serves as an approximate majorant for the perturbation series.
Introducing $p^*$ is not merely a technical device: for unbounded $q$, $\tp$ may fail to be comparable with $p$ in finite time. 
Finding an appropriate $p^*$ is essentially tantamount to estimating $\tp$, cf. \eqref{def:coeNstriv}, and may be hard, but in some cases, including ours, it is sufficient to chose  $\sp$ being a {\it dilation} of $p$.  
The $p^*$ majorization technique further involves an integral smallness condition for $p$, $q$ and $p^*$, which is implied by the familiar Kato-type conditions if $p$ and $p^*$ satisfy the 4G inequality.

In this paper we prove a 4G inequality for the transition density $p$ of the inverse Gaussian subordinator, including the $1/2$-stable subordinator. We reveal a wide class of unbounded Schr\"odinger potentials admissible for this $p$, and estimate the Schr\"odinger perturbations series
for $p$ using the framework of \cite{MR3200161}.
We thus extend the scope of the $p^*$ majorization technique for Schr\"odinger perturbations, beyond the transition densities of diffusion processes discussed in \cite{MR3200161}. We expect 4G to be valid quite generally, but at present it is even open for the $\alpha$-stable subordinators with $\alpha\neq 1/2$.
We note that the methods of \cite{MR3000465}, which make assumptions on potentials $q$ in terms of {\it bridges} (see also \cite{MR2457489}), fail for unbounded $q$ in this case. 
Namely, if $p$ is the transition density of a subordinator and $q$ is time-independent and unbounded, then $p$ and $\tilde p$ are never comparable, 
as proved in Section~\ref{sec:Kato}.
The results explain why we propose 4G and the framework of \cite{MR3200161} as a viable general approach to Schr\"odinger perturbations of transition densities by unbounded functions $q$.

The structure of the paper is as follows. Below in this section we give notation and preliminaries.
In Section~\ref{sec:p12s} we present 4G inequality and applications to Kato-type perturbations for the $1/2$-stable subordinator and the inverse Gaussian subordinator.
In Section~\ref{sec:Kato} we discuss unbounded perturbations $q$ of general subordinators.
In Lemma~\ref{cor:fst} and in Section~\ref{sec:uniq}
we discuss the connection of the considered integral operators to generators,
with focus on L\'evy-type generators.
In Remark~\ref{r:sp} we indicate extensions of our results to the case of signed $q$.

Let $X$ be an arbitrary set with a $\sigma$-algebra $\mathcal{M}$ and a (non-negative) $\sigma$-finite measure $m$ defined on $\mathcal{M}$. To simplify the notation we write $dz$ for $m(dz)$ in what follows. We also consider the 
$\sigma$-algebra $\mathcal{B} $ of Borel subsets  of $\RR$, and the Lebesgue measure, $du$, defined on $\RR$. The {\it space-time}, $\RR\times X$, is equipped with the $\sigma$-algebra $\mathcal{B}\times\mathcal{M}$ and the product measure $du\,dz=du\,m(dz)$.
We consider a {\it measurable transition density} $p$ on space-time, i.e.,
we assume that $p: \RR\times X \times\RR\times X\to [0,\infty]$
is
$\mathcal{B}\times\mathcal{M}\times\mathcal{B}\times\mathcal{M}$-measurable
and the following  Chapman-Kolmogorov equations hold for all $x,y \in X$  and $s<u<t$:
\begin{align}\label{assume1}
\int_{X} p(s,x,u,z)p(u,z,t,y)\,dz = p(s,x,t,y)\,.
\end{align}
All the functions considered below are assumed measurable on their respective domains.
We consider a (nonnegative and $\mathcal{B}\times\mathcal{M}$-measurable) function $q\colon \RR \times X \to [0,\infty]$.
The Schr{\"o}dinger perturbation $\tilde p$ of $p$ by $q$ is defined as
\begin{equation}\label{def:tp}
\tp(s,x,t,y)=\sum_{n=0}^{\infty}p_n(s,x,t,y)\,,
\end{equation}
where $p_{0}(s,x,t,y)=p(s,x,t,y)$ and, for $n=1,2,\ldots$,
\begin{align}\label{def:p_n}
p_n(s,x,t,y)&=\int_s^t \int_{X} p(s,x,u,z)q(u,z) p_{n-1}(u,z,t,y)\,dzdu\,.
\end{align}
The above is an explicit method of constructing new transition densities. In particular, $\tp$ satisfies the Chapman-Kolmogorov equations \cite[Lemma~2]{MR2457489}.
From \eqref{def:p_n} and the perturbation series 
\eqref{def:tp}
we get the perturbation formula:
\begin{equation}\label{eq:pfv}
\tp(s,x,t,y)=p(s,x,t,y)+\int_s^t \int_{X} p(s,x,u,z)q(u,z) \tp(u,z,t,y)\,dzdu.
\end{equation}
We similarly get the following variant,
\begin{equation}\label{eq:pf}
\tp(s,x,t,y)=p(s,x,t,y)+\int_s^t \int_{X} \tp(s,x,u,z)q(u,z) p(u,z,t,y)\,dzdu.
\end{equation}
Since $q\geq 0$, we trivially have $\tp \geq p$, so we focus on the upper bounds for $\tp$. These may be obtained under suitable conditions on $p_1$.
In \cite{MR3000465} (see also \cite{MR2457489}, \cite{MR2507445} and \cite[Lemma 3.1]{MR1978999}), the authors assume that for all $s<t$, $x,y \in X$,
\begin{equation}\label{eq:kk_1}
\int_s^t\int_X p(s,x,u,z)q(u,z)p(u,z,t,y)dzdu\leq [\eta +Q(s,t)]p(s,x,t,y),
\end{equation}
where $0\le \eta<\infty$ and  $Q$ is superadditive: $0\le Q(s,u)+Q(u,t)\le Q(s,t)$. The following estimates follow: for all $s<t$, $x,y\in X$,
\begin{equation}\label{eq:metgKpt}
\tp (s,x,t,y)\leq p (s,x,t,y){\left(\frac{1}{1-\eta}\right)}^{1+Q(s,t)/\eta},
\end{equation}
provided  $0<\eta<1$, and for $\eta=0$ we even have
\begin{equation}\label{eq:gpte4}
\tilde{p}(s,x,t,y)\leq p (s,x,t,y)e^{Q(s,t)}\,.
\end{equation}
The condition \eqref{eq:kk_1} may be considered as property of relative boundedness of $q$, or Miyadera-type condition for bridges \cite{MR1811962,MR2457489}.
It is convenient to use
\eqref{eq:kk_1}, e.g.,
for the transition density of the isotropic $\alpha$-stable L\'evy process with $\alpha\in (0,2)$, because the so-called 3G inequality holds in this case:
$$
p(s,x,u,z)\wedge p(u,z,t,y)\le c \ p(s,x,t,y), \qquad s<u<t, \ x,y,z \in \Rd.
$$
3G simplifies the verification of \eqref{eq:kk_1} allowing for a simple description of the acceptable growth of $q$, cf. \cite[Corollary~11]{MR2457489}, \cite[Section~4]{MR3000465}.
In general, however, condition \eqref{eq:kk_1} may be troublesome. For instance, the transition density of the Brownian motion fails to satisfy 3G and \eqref{eq:kk_1} is difficult to characterize in a simpler way. Moreover, as we see  below, for some transition densities  \eqref{eq:kk_1} holds for $q(u, z)=q(z)$ (i.e. time independent $q$) only if $q$ is bounded.
This explains the need for modifications of \cite{MR3000465}.
The approach of \cite{MR3200161} is based on the assumption that for all $s<t$, $x,y\in X$,
\begin{equation}\label{def:coeNs}
\int_s^t \int_{X} p(s,x,u,z)q(u,z) \sp(u,z,t,y)\,dzdu \leq \big[ \eta + Q(s,t)\big] \sp(s,x,t,y)\,.
\end{equation}
Here it is furthermore assumed that $0\le \eta<\infty$, $Q(s,t)$ is superadditive,
right-continuous  in $s$ and left-continuous in $t$ (in short: regular superadditive), and $\sp$ is a (majorizing)
transition density, i.e., there is a constant $C\ge 1$ such that for all $s<t$ and  $x,y\in X$,
\begin{equation}\label{eq:pCtp}
p(s,x,t,y) \leq C \sp (s,x,t,y)\,.
\end{equation}
The above assumptions are abbreviated to $q \in \pN(p,\sp,C,\eta,Q)$. By \cite[Theorem~1.1]{MR3200161}, if $q \in \pN(p,\sp,C,\eta,Q)$ with $\eta<1$,  then for every $\varepsilon\in(0,1-\eta)$,
\begin{equation}\label{ineq:thm1a}
\tp (s,x,t,y)\leq \sp (s,x,t,y)\left( \frac{C}{1-\eta-\varepsilon}\right)^{1+\frac{Q(s,t)}{\varepsilon}}\,,
\quad s<t,\; x,y\in X\,.
\end{equation}
For instance $\sp(s,x,t,y)=p(s/c,x,t/c,y)=c^d p(cs,cx,ct,cy)$ with $c\in (0,1)$ is convenient for the Gaussian kernel in $\Rd$ \cite{MR3200161}, and $Q(s,t)=\beta(t-s)$ with a constant $\beta\ge 0$ is a common choice.
In this work we use similar dilations to produce $\sp$. 

In principle,
\eqref{def:coeNs}
relaxes \eqref{eq:kk_1} and allows for more functions $q$.
This is seen in \cite{MR3200161} and again
in Section  \ref{sec:Kato} below, where we consider applications to transition densities of subordinators.
We should note that the flexibility comes at the expense of the sharpness
of the resulting estimate, as seen when comparing \eqref{eq:metgKpt} and \eqref{eq:gpte4} with \eqref{ineq:thm1a}.
Also, the methods of \cite{MR3200161} and the present paper are restricted to transition densities, while the methods of \cite{MR3000465} handle the more general so-called {\it forward} integral kernels. Last but not least, it may be cumbersome to point out $p^*$ suitable for $p$ and $q$, because this essentially requires guessing  
the rate of inflation of $\tp$.
In this connection we note that \eqref{eq:pf} trivially yields
\begin{equation}\label{def:coeNstriv}
\int_s^t \int_{X} p(s,x,u,z)\eta q(u,z) \tp(u,z,t,y)\,dz\,du \leq \eta\tp(s,x,t,y)\,.
\end{equation}
Thus, for perturbations of $p$ by $\eta q\ge 0$ with $0\le \eta<1$ one may take $\sp=\tp$, hence
estimating $\tp$ and finding an appropriate majorant $p^*$ are closely related problems.
Comparing to the approach of \cite{MR3000465} we finally note that $\sp$ should reflect the growth patterns of $\tp$, which $p$ is not always able to do.

We say that $q$
satisfies the parabolic Kato condition for $p$ if
\begin{align}\label{Kato1a}
\lim_{h\to 0^+}\sup_{s\in \RR, x\in X}\int_s^{s+h}\int_X p(s,x,u,z)q(u,z)\,dzdu=0\,,
\end{align}
and
\begin{align}\label{Kato1b}
\lim_{h\to 0^+}\sup_{t\in \RR, y\in X}\int_{t-h}^{t}\int_X p(u,z,t,y)q(u,z)\,dzdu=0\,,
\end{align}
cf. \cite[(29), (30)]{MR2457489}.
The relations between \eqref{Kato1a}, \eqref{Kato1b}, 3G and \eqref{eq:kk_1} is discussed in \cite[Lemma~9 and Corollary~11]{MR2457489} and \cite[(40), (7) and Lemma~5]{MR2457489}. 
Similar connections exist for \eqref{def:coeNs}, parabolic Kato conditions and 4G,
but we leave the details to the interested reader (see also the proof of Proposition~\ref{prop:Kato}).

Of particular interest here is the special case of convolution semigroups of probability measures $\{p_t\}_{t\ge 0}$ on $X=\Rd$, which are defined by the generating (L\'evy) triplets $(A,b,\nu)$ \cite{MR1739520}, and correspond to the generators
\begin{align}\label{formula:gen_Levy}
Lf(x) =& \frac{1}{2}\sum_{j,k=1}^d A_{j,k} \frac{\partial^2 f}{\partial x_j \partial x_k}(x)+\sum_{j=1}^d b_i \frac{\partial f}{\partial x_j}(x)\\ \nonumber
&+\int_{\RR^d}\left(f(x+y)-f(x) -\sum_{j=1}^d y_j\frac{\partial f}{\partial x_j} (x) 1_{|y|\le 1} (y) \right)\nu(dy)\,.
\end{align}
Namely we let $P_tf(x)=\int_{\Rd}f(z+x)p_t(dz)$, $t\geq 0$, and recall that $(P_t)_{t\geq 0}$  form a strongly continuous semigroup on $(C_0(\Rd),||\cdot||_{\infty})$, whose infinitesimal generator $L$ satisfies \eqref{formula:gen_Levy} for $f\in C_0^2(\Rd)$ and $x\in \Rd$. 
Furthermore for all $s\in\RR$, $x\in\RR^d$ and $\phi\in C_c^{\infty}(\RR\times\RR^d)$ (smooth compactly supported functions on space-time $\RR\times\RR^d$) we have
\begin{align}\label{eq:fsLp}
\int_s^{\infty} \int_{\RR^d} \Big[\partial_u \phi(u,x+z) + L\phi(u,x+z) \Big]p_{u-s}(dz)du = -\phi(s,x)\,.
\end{align}
The identity is essentially a consequence of the fundamental theorem of calculus. It is proved  in Section~\ref{sec:uniq} in the generality of strongly continuous operator semigroups.
We also 
provide a uniqueness result there.
A special case of $L$ is the Weyl derivative of order $1/2$ on 
the real line:
\begin{align}\label{eq:gen_sub_12}
\partial^{1/2} f(x)= \pi^{-1/2} \int_x^{\infty} f'(z) (z-x)^{-1/2}\,dz\,,\quad f\in C_c^1(\RR)\,.
\end{align}
We then have
\begin{align}\label{def:p12}
p_t(dz)=(4\pi)^{-1/2} tz^{-3/2}\exp\left\{-t^2/(4z) \right\}\1_{z>0}\,dz\,,
\end{align}
the distribution of the $1/2$-stable subordinator \cite{MR1739520} (also called the L\'{e}vy subordinator).
More generally, we let $\lambda\ge0$, $\delta>0$, $z\in \RR$, $t>0$, and 
\begin{equation}\label{def:p}
p(t,z)=(4\pi)^{-1/2}\delta t z^{-3/2}
\exp\left\{-\frac{(\delta t-2\sqrt{\lambda}z)^2}{4z}\right\}\1_{z>0}\,.
\end{equation}
We note that $p(t,z)$ is the density function
of the distribution of the inverse Gaussian subordinator
$\xi_t=\inf\{s>0: \, B_{s}+\sqrt{2\lambda} s=t\delta/\sqrt{2}   \}$, where $B$ is the standard
one-dimensional Brownian motion,
cf. \cite[Example~1.3.21]{MR2512800} and \cite[Table~4.4]{MR2042661}.
Alternatively $p$
may be obtained from the density function of the $1/2$-stable subordinator by
the Esscher transform  and time rescaling, see \cite[Example 33.15]{MR1739520} or \cite[Sec.~4.4.2]{MR2042661}.
Accordingly, the L\'{e}vy measure $\mu$ of the inverse Gaussian subordinator  is obtained by the exponential tilting of the L\'{e}vy measure $\nu$ of the $1/2$-stable subordinator, where
\begin{align*}
\nu(dy)= \frac{1}{2\sqrt{\pi}} \,y^{-3/2} 1_{y>0}\,dy\quad\text{ and }\quad \mu(dy)=\delta e^{-\lambda y}\nu(dy).
\end{align*}
The generator corresponding to the inverse Gaussian subordinator
is calculated  for $f\in C_c^1(\RR)$ as 
\begin{equation}\label{eq:giG}
Lf(x)=\frac{\delta}{2\sqrt{\pi}} \int_x^{\infty} f'(z)\,\Gamma_{\lambda}(-1/2,z-x)\,dz\,.
\end{equation}
Here $\Gamma_{\lambda}(a,z)=\int_z^{\infty} e^{-\lambda y} y^{a-1}\,dy$ for $\lambda, z>0$, $a\in\RR$, is the incomplete gamma function. For the readers's convenience we prove \eqref{eq:gen_sub_12} and \eqref{eq:giG} in Section~\ref{sec:uniq}. Some further discussion can be found in  \cite{Yana}.
We also note that 
the Laplace exponent of $\xi_t$ is 
$u\mapsto\delta(\sqrt{u+\lambda}-\sqrt{\lambda})$,
see, e.g., \cite[Example~1.3.21]{MR2512800}, \cite[Example~8.11 and 33.15]{MR1739520}.

\section{4G inequality for the inverse Gaussian  subordinator}\label{sec:p12s}
Our main goal is to give conditions for and discuss consequences of \eqref{ineq:thm1a}.
Let $\lambda\ge0$ and $\delta>0$. Using \eqref{def:p} we define 
$$p(s,x,t,y)=p(t-s,y-x),$$
if $s< t$ and  $x,y\in \RR$, and we let $p=0$ otherwise. It is a transition density on $X=\RR$ with respect to  the Lebesgue measure.
We observe that 3G inequality does not hold for $p$. Indeed, if
$u-s=t-u=z-x=y-z=\theta$, then
\begin{align*}
p(s,x,u,z)\land p(u,z,t,y)&= (4\pi)^{-1/2} \delta\, \theta^{-1/2} \exp\left\{- \,\theta (\delta-\sqrt{\lambda})^2/4\right\},\\
p(s,x,t,y)&= (4\pi)^{-1/2} \delta \, (2\theta)^{-1/2} \exp\left\{-2 \,\theta (\delta-\sqrt{\lambda})^2/4\right\},
\end{align*}
and the second expression decays exponentially faster as $\theta\to \infty$.
For $c>0$ we consider auxiliary (inverse Gaussian) transition density
\begin{equation}\label{def:pc}
\rho_c(s,x,t,y):=cp(c(t-s),c(y-x)).
\end{equation}
In view toward \eqref{eq:pCtp} we note that for $0<a<b$,
\begin{align}\label{ineq:pb}
\rho_b(s,x,t,y)\leq \left(b/a \right)^{1/2}\rho_a(s,x,t,y)\,.
\end{align}
We shall consider the Schr\"odinger perturbation $\tp$ of $p=\rho_1$ by $q$.
Clearly, if $q\in\pN(\rho_1,\rho_a,(1/a)^{1/2},Q,\eta)$, with $0<a<1$, $\eta\in[0,1)$, then $\tp$ is finite, in fact it satisfies \eqref{ineq:thm1a}. Here is a connection to generators. 
\begin{lemma}\label{cor:fst}
If $q\in\pN(\rho_1,\rho_a,(1/a)^{1/2},Q,\eta)$, where $0<a<1$, $\eta\in[0,1)$, then
\begin{align*}
\int_s^{\infty} \int_{\RR} \tp(s,x,u,z)\Big[ \partial_u \phi(u,z)+ L\phi(u,z) + q(u,z)\phi(u,z) \Big]\,dzdu=-\phi(s,x)\,,
\end{align*}
for $\phi\in C_c^{\infty}(\RR\times \RR)$, where $L$ is given by \eqref{eq:giG}.
\end{lemma}

\begin{proof}
We will follow the proof of \cite[Lemma 4]{MR3000465} with some modifications. 
We define integral operators
\begin{align*}
Pf(s,x)&=\int_s^\infty \int_\RR p(s,x,t,y)f(t,y)dydt,\\
qf(s,x)&=q(s,x)f(s,x),\\
\tP f(s,x)&=\int_s^\infty \int_\RR \tp(s,x,t,y)f(t,y)dydt,\\
\Ps f(s,x)&=\int_s^\infty \int_\RR \rho_a(s,x,t,y)f(t,y)dydt,
\end{align*}
for $s,x\in \RR$ and jointly measurable and nonnegative or absolutely integrable  functions $f:\RR\times \RR\to \RR$.
By \eqref{eq:pfv} and \eqref{eq:pf},
\begin{equation}\label{eq:pfc}
\tP=P+Pq\tP=P+\tP qP,
\end{equation}
hence $\tP q P=Pq \tP$.
Let $\phi\in C^\infty_c(\RR\times \RR)$ and $\psi=\partial_s \phi+L\phi$. By \eqref{eq:fsLp},
\begin{align*}
\int_s^{\infty} \int_{\RR} p(s,x,u,z)\Big[ \partial_u \phi(u,z)+ L\phi(u,z)  \Big]\,dzdu=-\phi(s,x)\,.
\end{align*}
In short, $P\psi=-\phi$. For clarity, since $\psi$ is bounded \cite[p.~211]{MR1739520} and $\psi(s,x)=0$ if $|s|$ is large, 
we have $P|\psi|<\infty$.
By \eqref{eq:pfc} and Fubini's theorem,
\begin{align*}
\tP (\psi+q\phi)= (P+\tP q P)\psi+\tP q \phi
=-\phi +\tP q P \psi + \tP q (-P\psi)= -\phi\,.
\end{align*}
The identity is precisely the claim of the lemma, but 
we need to verify the absolute convergence of the integrals above.
Since $q\in\pN(\rho_1,\rho_a,(1/a)^{1/2},Q,\eta)$, we have
$Pq\Ps\le c \Ps$ and $\tP\le c \Ps$, when applied to nonnegative functions in a bounded time horizon, cf. \eqref{def:coeNs} and \eqref{ineq:thm1a}.
It follows that in bounded time, 
$$
\tP q P |\psi|=Pq\tP|\psi|\le c Pq \Ps |\psi|\le c \Ps |\psi|<\infty.
$$
Since $|\phi|\le P|\psi|$, we get $\tP q |\phi|<\infty$. 
The proof is complete.
\end{proof}

Following \cite[Theorem~1.1]{MR1883198} and \cite{MR3200161}, the identity in the statement of Lemma~\ref{cor:fst} is interpreted by saying that $\tp$ is a fundamental solution of $\partial_s+L+q$ or, in short, for $L+q$. The identity also means that $\tp$ as integral operator is the left inverse of $\partial_s+L+q$. We refer to \cite[Remark~4.10]{MR3200161} for further discussion.

We point yet another aspect of the relationship between $\tp$ and $L+q$.
By Lemma~\ref{cor:fst} and Chapman-Kolmogorov, for $\phi\in C_c^{\infty}(\RR\times \RR)$ we obtain
\begin{align*}
&\int_s^{t}\int_{\RR} \tp (s,x,u,z) \Big[ \partial_u + L + q(u,z) \Big]\phi(u,z)\,dzdu\\
&=\int_s^{\infty}\int_{\RR} \tp (s,x,u,z) \Big[ \partial_u + L + q(u,z)\Big]\phi(u,z)\,dzdu\\
&-\int_t^{\infty}\int_{\RR} \tp (s,x,u,z) \Big[ \partial_u + L + q(u,z) \Big]\phi(u,z)\,dzdu\\
&=-\phi(s,x)
- \int_{\RR}\tp(s,x,t,w)\int_t^{\infty}\int_{\RR} \tp (t,w,u,z)
\times \\
&\hspace{100pt}\times
\Big[ \partial_u + L + q(u,z)\Big]\phi(u,z)\,dzdu\,dw\\
&=-\phi(s,x)
+\int_{\RR}\tp(s,x,t,w)\bigg[-\int_t^{\infty}\int_{\RR} \tp (t,w,u,z) \times\\
&\hspace{100pt}\times \Big[ \partial_u + L + q(u,z) \Big]\phi(u,z)\,dzdu\,dw\bigg]\\
&=\int_{\RR}\tp(s,x,t,w)\phi(t,w)dw-\phi(s,x)\,,\quad s<t,\ x\in \RR\,,
\end{align*}
and by choosing $\phi$ constant in time on $(s,t)$, for $\varphi\in C_c^{\infty}(\RR)$ we get
\begin{align*}
\int_{\RR}\tp (s,x,t,z)\varphi(z)\,dz -\varphi(x)=\int_s^{t}\int_{\RR} \tp (s,x,u,z) \Big[L\varphi(z) + q(u,z)\varphi(z) \Big]\,dzdu\,.
\end{align*}
The identity is an analogue of classical formulas for strongly continuous operator semigroups, and so is \eqref{eq:pf}. 
Further discussion of the connection to generators 
is given in Section~\ref{sec:uniq}.

We now investigate the class $\pN(\rho_b,\rho_a,(b/a)^{1/2},\eta,Q)$, where  $0<a<b$, namely we propose conditions sufficient for \eqref{def:coeNs}. 
We first recall results of \cite[Section 3]{MR3200161} on the Gaussian kernel
\begin{align}\label{def:gc}
g_c(s,\dd{x},t,\dd{y}):=[4\pi (t-s)/c ]^{-d/2}\exp \left\{-|\dd{y}-\dd{x}|^2/[4(t-s)/c] \right\},
\end{align}
where $c>0$, $0<s<t$, $\dd{x},\dd{y}\in\Rd$ and $d\in\NN$. We denote
$$
l(\alpha)=\max_{\tau\geq \alpha \vee 1/\alpha} \left[\ln(1+\tau) - \frac{\tau-\alpha}{1+\tau} \ln(\alpha\tau) \right],
$$
and for $0<a<b$ we let $M=\left( \frac{b}{b-a}\right)^{d/2}\exp\left[ \frac{d}{2}l(\frac{a}{b-a}) \right]$. Then we have
\begin{align}\label{ineq:4Gg}
\frac{g_b (s,\dd{x},u,\dd{z}) g_a(u,\dd{z},t,\dd{y})}{g_a(s,\dd{x},t,\dd{y})} \leq M [g_{b-a}(s,\dd{x},u,\dd{z})\vee g_a(u,\dd{z},t,\dd{y})]\,,
\end{align}
where $s<u<t$ and $\dd{x},\dd{z},\dd{y}\in\Rd$
\cite[Theorem 1.3 and Remark 3.2]{MR3200161}.  Moreover,  $M$ is the optimal constant in \eqref{ineq:4Gg}, and  if $b/a\le 1+ e^{-1/2}$, then $M=(1-a/b)^{-d}$.
This 4G inequality is used in \cite{MR3200161} to obtain Gaussian estimates for 
Schr\"odinger perturbations of transition densities of the second order parabolic differential operators.  In this section we prove a similar inequality for the transition density $\rho_c$ defined in \eqref{def:pc}.
\begin{theorem}[4G]\label{thm:4Gp}
Let $0<a<b$.
For all $s<u<t$ and $x<z<y$,
\begin{align}\label{ineq:4Gs}
\!\!\!\!\!\rho_b(s,x,u,z)\rho_a(u,z,t,y)\leq D \Big[ \rho_{b-a}(s,x,u,z)\vee \rho_a(u,z,t,y) \Big] \rho_a(s,x,t,y)\,
\end{align}
holds with  $D=\left( \frac{b}{b-a}\right)^{3/2}\exp\left[ \frac{3}{2}L\left(\frac{a}{b-a}\right) \right]$.
\end{theorem}
\begin{proof}
We denote $\dd{r}=(r,0,0)\in\RR^3$ for $r\in\RR$. For $c>0$, $s<t$, $x<y$,
\begin{align*}
\rho_c (s,x,t,y)=(4\pi\delta (t-s)/c)\, g_c (x,\delta\dd{s}-2\sqrt{\lambda}\dd{x},y,\delta\dd{t}-2\sqrt{\lambda}\dd{y})\,.
\end{align*}
By \eqref{ineq:4Gg} for all $s<u<t$ and $x<z<y$ we have
\begin{align*}
\rho_b &(s,x,u,z) \rho_a (u,z,t,y)= \frac{(4\pi\delta)^2 (u-s)(t-u)}{ab}\times\\
&\quad\times g_b(x,\delta\dd{s}-2\sqrt{\lambda}\dd{x},z,\delta\dd{u}-2\sqrt{\lambda}\dd{z}) g_a(z,\delta\dd{u}-2\sqrt{\lambda}\dd{z},y,\delta\dd{t}-2\sqrt{\lambda}\dd{y})\\
&\leq \frac{(4\pi\delta)^2 (u-s)(t-u)}{ab} D\,g_a(x,\delta\dd{s}-2\sqrt{\lambda}\dd{x},y,\delta\dd{t}-2\sqrt{\lambda}\dd{y})\times\\
&\quad\times  [g_{b-a}(x,\delta\dd{s}-2\sqrt{\lambda}\dd{x},z,\delta\dd{u}-2\sqrt{\lambda}\dd{z}) \vee g_{a}(z,\delta\dd{u}-2\sqrt{\lambda}\dd{z},y,\delta\dd{t}-2\sqrt{\lambda}\dd{y})]\\
&= D \left[ \left(\frac{t-u}{t-s}\,\frac{b-a}{b}\right) \rho_{b-a}(s,x,u,z) \vee \left(\frac{u-s}{t-s}\,\frac{a}{b} \right) \rho_a(u,z,t,y)\right] \rho_a (s,x,t,y)\\
&\leq D \Big[ \rho_{b-a}(s,x,u,z) \vee \rho_a(u,z,t,y)\Big]\, \rho_a (s,x,t,y)\,.
\end{align*}
\end{proof}
We are ready to give sufficient  conditions for \eqref{def:coeNs}. First comes an immediate consequence of Theorem \ref{thm:4Gp}.
\begin{corollary}
Assume that for all $s<t$, $x<y$,
\begin{align*}
D \int_s^t \int_{\RR} \Big[ \rho_{b-a}(s,x,u,z)+ \rho_a(u,z,t,y) \Big]q(u,z) \,dzdu \leq \eta + Q(s,t)\,.
\end{align*}
Then $q\in \pN(\rho_b, \rho_a,(b/a)^{1/2}, \eta,Q)$.
\end{corollary}
\noindent
Motivated by \eqref{Kato1a} and \eqref{Kato1b} for 
$c,h>0$ we next define
\begin{align*}
N_h^c(
q) =&\sup_{s,x} \int_s^{s+h}\int_{\RR} \rho_c(s,x,u,z)q(u,z)\,dzdu\\
+&\sup_{t,y}\int_{t-h}^t\int_{\RR}\rho_c(u,z,t,y)q(u,z)\,dzdu.
\end{align*}

\begin{proposition}\label{prop:Kato}
Let $0<a<b$ and $D'=\left(\frac{b-a}{a}\vee \frac{a}{b-a} \right)^{1/2}D$.
If 
\begin{align}\label{ineq:N_male}
N_{h}^{(b-a)\land a}(q)\leq \eta /D'
\end{align}
for some $0<h\le \infty$, $0\le \eta<\infty$, then for $Q(s,t)=\eta (t-s)/h$ we have
$$
q\in\pN(\rho_b,\rho_a,(b/a)^{1/2},\eta,Q)\,.
$$
\end{proposition}
\begin{proof}
Follow \cite[p.~165]{MR3200161}.
\end{proof}
The condition $\lim_{h\to 0} N_h^c(q) =0$ 
defines the parabolic Kato class
for $\rho_c$, cf. Section~\ref{sec:intro}, and if it is satisfied, then Proposition~\ref{prop:Kato} applies.
A~thorough discussion of the Kato condition for arbitrary L\'evy processes on $\Rd$ is given in 
\cite{2015arXiv150305747G}. For the considered inverse Gaussian subordinator \eqref{def:p}, including the $1/2$-stable subordinator, if $q(u,z)=q(z)$  is time-independent, then the Kato condition is equivalent to
$$
\lim_{r\to 0^+} \sup_{x\in \RR} \int_{x-r}^{x+r} q(z)|z-x|^{-1/2}dz=0.
$$
We refer 
to \cite[Example~3]{2015arXiv150305747G} for the result.
A characteristic example here is $q(z)=|z|^{\varepsilon-1/2}$ for $\varepsilon \in (0,1/2]$.

In the remainder of this section
we focus on the case $\lambda=0$ and $\delta=1$ in \eqref{def:p}, i.e., on the density of the $1/2$-stable subordinator, with emphasis on
honest constants in estimates.
\begin{example}\label{ex:1}
\rm
We consider $q(u,z)=q(z)$ on $\RR$. Let 
$r>2$  and $q\in L^r(\RR)$.  Observe that for all $s<u$, $x\in \RR$ and $c>0$,
\begin{align*}
\int_{\RR} \rho_{c}(s,x,u,z)^{\sigma}\,dz= \frac{c'_{\sigma}}{c^{\sigma-1}} (u-s)^{-2(\sigma-1)} \,, \quad \sigma\geq 1 \,,
\end{align*}
where $c'_{\sigma}=(4\pi)^{-\sigma/2}(4/\sigma)^{3\sigma/2-1}\Gamma(3\sigma/2-1)
\leq \big[(4\pi)^{-1/2}(6/e)^{3/2}\big]^{\sigma-1}$. By H\"older's inequality, for $h>0$,
\begin{align*}
\sup_{s,x}&\int_s^{s+h}\int_{\RR}\rho_c(s,x,u,z)q(
z)\,dzdu\\
&\leq  \sup_{s,x}\int_s^{s+h} (u-s)^{-2/r}\,du\,\, \frac{\left( c'_{r/(r-1)}\right)^{(r-1)/r}}{c^{1/r}}||q||_r\\
&= h^{1-2/r}\, \left[ \left(c'_{r/(r-1)}\right)^{(r-1)/r}c^{-1/r}||q||_r/(1-2/r)\right]\,.
\end{align*}
Thus for every $c>0$,
\begin{align}\label{ineq:jawne}
N_h^c(q) \leq h^{1-2/r}\,  2 \left[ \frac{\left(c'_{r/(r-1)}\right)^{(r-1)/r}}{(1-2/r)\,c^{1/r}}||q||_r\right] \to 0\,, \quad \mbox{if}\quad h\to 0^+\,.
\end{align}
Notice also that $\left( c'_{r/(r-1)}\right)^{(r-1)/r}\leq \big[(4\pi)^{-1/2}(6/e)^{3/2}\big]^{1/r}$.
Finally, by Proposition \ref{prop:Kato} for all $0<a<b$ we obtain
\begin{align*}
q\in \pN(\rho_b,\rho_a,(b/a)^{1/2},\eta,Q)\,,
\end{align*}
with arbitrary $\eta>0$ and $Q(s,t)=\eta (t-s)/h$,
provided $h$ satisfies
$$
h^{1-2/r}\,  \frac{2D \left(\frac{b-a}{a}\vee \frac{a}{b-a} \right)^{1/2}}{(1-2/r)} \left[ \frac{ (4\pi)^{-1/2}(6/e)^{3/2}}{(b-a)\land a}\right]^{1/r}||q||_r =\eta\,.
$$
Indeed, \eqref{ineq:jawne} implies \eqref{ineq:N_male}.
\end{example}

We keep investigating the class $\pN(\rho_b,\rho_a,(b/a)^{1/2},\eta,Q)$ by estimating $N_h^c(q)$ for time-independent $q(u,z)=q(z)$. We first prove an auxiliary lemma for the general $\alpha$-stable subordinator, with $\alpha\in(0,1)$. Let 
\begin{align*}
I_{\varepsilon}(q)= \sup_{x\in\RR}\int_{|x-z|<\varepsilon} \frac{q(z)}{|x-z|^{1-\alpha}} \,dz\,,\qquad \varepsilon>0\,.
\end{align*}
Let $\gamma(t,z)$ be the density of the $\alpha$-stable subordinator, in particular, $\gamma(t,z)=0$ for $z\leq 0$ and
$$
\int_0^\infty e^{-uz}\gamma(t,z)\,dz=e^{-tu^{\alpha}}, \qquad u\ge 0,\quad t>0\,.
$$
\begin{lemma} For all $c,r,\tau>0$ and $0<\alpha<1$,
\begin{align*}
\sup_{s\in\RR, x\in\RR}\int_s^{s+\tau} \int_{\RR} \gamma_c(s,x,u,z)q(z)\,dzdu\leq \left( \frac1{c^{1-\alpha}\Gamma(\alpha)}+ \frac{2\tau}{r^{\alpha}} \right) I_r(q)\,,
\end{align*}
where $\gamma_c(s,x,t,y)=c\,\gamma(c(t-s),c(y-x))=\gamma(c^{1-\alpha}(t-s),y-x)$.
\end{lemma}
\begin{proof}

Let $c>0$, $k(x)=\int_0^{\tau} \gamma_c(0,0,u,|x|)du$,
$K(x)=\int_0^{\infty}\gamma_c(0,0,u,|x|)du= |x|^{\alpha-1}/(c^{1-\alpha}\Gamma(\alpha))$,
$c_1=\int_\RR k(x)dx=2\tau$ and $c_2= rK(r)=r^{\alpha}/(c^{1-\alpha}\Gamma(\alpha))$.
By scaling, $\gamma_c(0,0,u,|x|)=|x|^{-1}\gamma_c(0,0,|x|^{-\alpha}u,1)$. By a change of variables, $k$ is  symmetrically decreasing. The result then follows from \cite[Lemma 4.2]{MR3200161}.\end{proof}
A direct consequence is that for every $\alpha$-stable subordinator and for all $s<t$, $x<y$ and $h>0$ we have
\begin{align*}
\int_s^t \int_{\RR} &\Big[ \gamma_{b-a}(s,x,u,z)+ \gamma_a(u,z,t,y) \Big]q(z)\,dzdu \\
&\le I_{h^{1/\alpha}}(q)\left[\frac1{\Gamma(\alpha)} \frac{a^{1-\alpha}+(b-a)^{1-\alpha}}{[a(b-a)]^{1-\alpha}}+  \frac{4(t-s)}{h}\right] \,.
\end{align*}
For $\alpha=1/2$ we may use Theorem \ref{thm:4Gp} to get
for all $s<t$, $x<y$ and $h>0$,
\begin{align*}
&\int_s^t \int_{\RR} \rho_b(s,x,u,z)q(z)\rho_a(u,z,t,y)\,dzdu \\
&\le D I_{h^2}(q)\left[\frac1{\Gamma(1/2)} \frac{\sqrt{a}+\sqrt{b-a}}{\sqrt{a(b-a)}}+  \frac{4(t-s)}{h}\right] \rho_a(s,x,t,y)\,.
\end{align*}
\begin{corollary}
Let $q\colon\RR\to\RR$ be such that $I_{h^2}(q)<\infty$ for some $h>0$.
Then $q\in\pN(\rho_b,\rho_a,(b/a)^{1/2},\eta,Q)$ with
\begin{align*}
&\eta= D I_{h^2}(q)\left(\sqrt{a}+\sqrt{b-a}\right)/\left(\Gamma(1/2) \sqrt{a(b-a)}\right),\\
&Q(s,t)= 4D I_{h^2}(q)  (t-s)/h\,.
\end{align*}
\end{corollary}
Summarizing this section, we see that 4G for the inverse Gaussian subordinator yields \eqref{def:coeNs} for a large class of functions $q$ characterized by simpler Kato-type conditions, and then $\tp$ satisfies \eqref{ineq:thm1a} and Lemma~\ref{cor:fst}.

\section{Relative boundedness for subordinators with transition density}\label{sec:Kato}

In this section we consider a
general subordinator with transition density $p$.
Thus, $p$ is space-time homogeneous, $p(s,x,t,y)=0$ whenever $t\leq s$ or $y\leq x$,
and $p(s,x,t,y)> 0$ otherwise.
We first discuss time-independent functions $q$, aiming at the condition~\eqref{eq:kk_1}.

We denote, as usual, $||f||_{\infty}= {\rm ess} \sup_{x\in\RR} |f(x)|$.
Let functions $(\phi_j)_{j\in\NN}$ be an approximation to identity in $L^1(\RR)$, that is real-valued on $\RR$ with the following properties:
\begin{align}
&\label{pr:2} \phi_j \geq 0 \mbox{\quad and\quad} \int_{\RR}\phi_j(z) dz=1\,,\\
&\label{pr:4} \forall_{\delta>0} \exists_{j_0\in\NN} \forall_{j\geq j_0} \quad\mbox{supp}(\phi_j) \subset (-\delta,\delta).
\end{align}

\begin{lemma}\label{lem:identity}
Let $f\in L^1_{loc}(\RR)$. If\, $\sup_{n\in\NN}||\phi_n * f ||_{\infty}<\infty$, then $f\in L^{\infty}(\RR)$.
\end{lemma}
\begin{proof}
Let $0<\delta<R$ and $M=\sup_{n\in\NN}||\phi_n * f ||_{\infty}$. Choose $j_0\in\NN$ according to \eqref{pr:4}. Since the functions $f1_{|z|< R} * \phi_n$ converge to $f1_{|x|<R}\in L^1(\RR)$ in the $L^1$ norm, a subsequence $f1_{|z|< R} * \phi_{n_k}$ converges almost surely to $f1_{|x|<R}$. For $n_k \geq j_0$,
\begin{align*}
f1_{|z|< R} * \phi_{n_k}(x) = f * \phi_{n_k}(x)\,,\qquad {\rm if} \quad |x|<R-\delta.
\end{align*}
Thus for almost all $|x|<R-\delta$,
\begin{align*}
|f(x)|=\lim_{k\to \infty} |f*\phi_{n_k}|\leq M\,.
\end{align*}
Therefore $|f(x)|\leq M$ for almost all $x\in\RR$.
\end{proof}
\begin{lemma}\label{lem:L1loc}
Assume that for some $s<t$ and all $x\in\RR$,
\begin{align*}
\int_s^t \int_{\RR} p(s,x,u,z)q(z)\,dzdu \leq M\,.
\end{align*}
Then $q\in L_{loc}^1(\RR)$.
\end{lemma}
\begin{proof}
Let $\varphi\in C_0(\RR)$ be such that $\varphi\geq 0$, $\varphi =1$ on $[0,1/2]$ and $\int_{\RR}\varphi(x)\,dx=1$. For arbitrary fixed $x_0\in\RR$ we have
\begin{align*}
M&\geq \int_s^t \int_{\RR} \int_{\RR} \varphi(x_0-x)p(s,x,u,z)dx\,q(z)\,dzdu\\
&=\int_s^t \int_{\RR} P_{u-s}\,\varphi (x_0-z) q(z)\,dzdu \geq (\varepsilon/2)\int_{x_0-1/2}^{x_0} q(z)\,dz\,,
\end{align*}
where $0< \varepsilon\leq t-s$ is such that $||P_u \varphi - \varphi ||_{\infty}\leq 1/2$ for $u\leq \varepsilon$.
\end{proof}
Lemma~\ref{lem:L1loc} is generalized to arbitrary L\'evy processes in $\Rd$ \cite[Lemma~3.7]{2015arXiv150305747G}.
\begin{theorem}
Assume that for some $s<t$,
\begin{align*}
\sup_{x<y}  \int_s^t \int_{\RR} \frac{p(s,x,u,z)p(u,z,t,y)}{p(s,x,t,y)} q(z)\,dzdu <\infty\,.
\end{align*}
Then $q\in L^{\infty}(\RR)$.
\end{theorem}
\begin{proof}
By the assumption there is $M'>0$ such that for  some fixed $s<t$,
\begin{align*}
\int_s^t \int_{\RR} \frac{p(s,x,u,z)p(u,z,t,y)}{p(s,x,t,y)} q(z)\,dzdu \leq M'\,,\qquad  x<y\,.
\end{align*}
By Lemma \ref{lem:L1loc}, $q\in L^1_{loc}(\RR)$. For $s<t$ and $n\in\NN$, we let
\begin{align*}
\phi_n(z)= \frac{1}{t-s}\int_s^t\frac{p(s,-1/n,u,-z)p(u,-z,t,1/n)}{p(s,-1/n,t,1/n)}\,du\,,\qquad |z|<1/n\,,
\end{align*}
and $\phi_n(z)=0$ for $|z|\geq 1/n$. Clearly,
$\phi_n$ satisfies conditions \eqref{pr:2} and \eqref{pr:4}. Furthermore, for all $x\in\RR$,
\begin{align*}
\phi_n* q(x) = \frac{1}{t-s}\int_s^t \int_{\RR} \frac{p(s,x-1/n,u,z)p(u,z,t,x+1/n)}{p(s,x-1/n,t,x+1/n)}q(z)\,dzdu\,.
\end{align*}
Thus, $\sup_{n\in\NN} ||\phi_n* q||_{\infty}\leq M'/(t-s)=M<\infty$. Lemma \ref{lem:identity} ends the proof.
\end{proof}

\begin{corollary}\label{cor:2}
Let $q(u,z)=q(z)$.
Then $q$ satisfies \eqref{eq:kk_1} if and only if $||q||_\infty<\infty$.
If there are $s<t$ and $C<\infty$ such that
$
\tp (s,x,t,y)\leq C\, p(s,x,t,y)
$
for all $x<y$, then 
$||q||_\infty<\infty$.
\end{corollary}
Corollary \ref{cor:2}  shows that the methods of \cite{MR3000465} 
cannot deliver estimates of
Schr\"odinger perturbations of transition densities $p$ of subordinators by unbounded time-independent $q$.
In contrast, we saw in Section~\ref{sec:p12s} that the methods based on  majorants $p^*$ and 4G inequality handle such situations.

If we allow $q$ to depend on time, 
the statements of the corollary are no longer valid. Indeed, let $q(u,z)=u_+^{-1/2}$, where $u_+=u\vee 0$. Then for all $s<t$ and $x<y$ and  transition densities $p$,
\begin{align*}
\int_s^t \int_{X} p(s,x,u,z)q(u,z)p(u,z,t,y)\,dzdu \leq 2 (t_+-s_+)^{1/2} p(s,x,t,y)\,.
\end{align*}
We see that this unbounded $q$ yields \eqref{eq:kk_1} and \eqref{eq:metgKpt}
for every $p$.

The next example builds on the ideas proposed in \cite[Example 4]{MR2507445}.
\begin{example}
Consider the second term $p_1$ of the perturbation series \eqref{def:tp} for $\tp$. Let 
\begin{align*}
\sup_{\substack{s\leq u \leq t\\ x \leq z \leq y}} q(u,z)\leq \eta/(t-s) \,,
\end{align*}
for some $\eta\geq 0$ and for all $s<t$, $x<y$ such that $(s,x),(t,y)\in F:=\{(u,z)\colon q(u,z)>0\}$.
Then we claim that for all $s<t$ and $x<y$,
\begin{align}\label{ineq:exunif}
p_1(s,x,t,y)\leq \eta\, p(s,x,t,y)\,.
\end{align}
For the proof we  consider a Borel non-decreasing function $\omega\colon [s,t] \to \RR$, $s<t$, such that $\omega(s)=x<y=\omega(t)$, and let
 $T(\omega)=\{ u \colon s\leq u\leq t,\,(u,\omega(u))\in F \}$.
If $T(\omega)$ is empty,
then
$$
\int_s^t q(u,\omega(u))\,du =0\leq \eta\,.
$$
Otherwise we consider $\sigma=\inf\{u\colon u\in T(\omega) \}$ and $\tau=\sup\{ u\colon u\in T(\omega)\}$. There are $s_n\le t_n$ such that $(s_n, \omega(s_n)), (t_n, \omega(t_n))\in F$, $s_n \downarrow \sigma$ and $t_n\uparrow \tau$, hence
\begin{align*}
\int_s^t q(u,\omega(u))\,du &= \int_{\sigma}^{\tau} q(u,\omega(u))\,du =\lim_{n\to \infty} \int_{s_n}^{t_n} q(u,\omega(u))\,du\\
&\leq \lim_{n\to \infty} (t_n-s_n) \!\!\!\!\!\!\!  \sup_{\substack{s_n\leq u \leq t_n\\ \omega(s_n) \leq z \leq \omega(t_n)}}\!\!\!\!\!\!\!\! q(u,z)  \leq \eta\,.
\end{align*}
Finally, let $\{Y_u\}_{u\geq 0}$ be the subordinator. Given $s<t$, $x<y$ we denote by $\{Z_u\}_{s\leq u\leq t}$ the bridge corresponding to $\{Y_u\}_{u\geq 0}$, which starts from $x$ at time $s$ and reaches $y$ at time $t$.
Since the trajectories of $\{Z_u\}_{u\geq 0}$ are almost surely non-decreasing we have for all $s<t$, $x<y$,
\begin{align*}
p_1(s,x,t,y)/p(s,x,t,y)&=\mathbb{E}_{s,x}^{t,y}\left[\int_s^t q(u,Z_u)\right] du\leq \mathbb{E}_{s,x}^{t,y}\big[\, \eta\, \big] =\eta\,,
\end{align*}
as claimed.

Typical applications are $q(u,z)=\eta z 1_{(0,1/u)}(z)$, cf. \cite[Example 4]{MR2507445}, and $q(u,z)=\eta z^2 1_{F}(u,z)$, where $F=\bigcup_{n=1}^{\infty} \left(1/(n+1),n \right)\times \left(n-1,n \right)$. Both functions tend to infinity when time goes to zero and the space variable grows correspondingly.
\end{example}

We next 
show that the estimate \eqref{ineq:exunif} 
cannot be improved.
\begin{example}
We 
define $q(u,z)=\eta z 1_{(0,1/u)}(z)$, $\eta>0$. Let $\varepsilon<\eta$. We claim that there is no superadditive $Q$ such that
\begin{align}\label{ineq:exfail}
p_1(s,x,t,y)\leq \big[\varepsilon + Q(s,t)\big] p(s,x,t,y)\,.
\end{align}
Indeed, 
by \cite[Lemma 5.3]{MR3200161} 
we may assume that $Q$ is regular
superadditive. Thus there is $t$ such that $\big[\varepsilon + Q(0,t)\big]<(\varepsilon+\eta)/2$. On the other hand for $x:=(1+\varepsilon/\eta)/(2t)<y:=1/t$ we have
\begin{align*}
p_1(s,x,t,y)&=\int_0^t\int_x^y p(s,x,u,z)q(u,z)p(u,z,t,y)\,dzdu\\
&\geq \eta x \int_0^t \int_x^y p(s,x,u,z)p(u,z,t,y)\,dzdu\\
&\geq \eta x t \,p(s,x,t,y)= \big[(\eta +\varepsilon)/2\big]\, p(s,x,t,y) \,,
\end{align*}
which is a contradiction.
\end{example}

\section{Appendix}\label{sec:uniq}
In this section we prove \eqref{eq:fsLp} and its analogues in the setting of general semigroup theory.
We consider a Banach space $(Y,||\cdot||)$. Let $T=(T_t)_{t\ge0}$ be a strongly continuous semigroup of linear operators on $Y$. Let $L$  be the corresponding infinitesimal generator with domain $D(L)$ \cite[IX]{MR617913}.
\begin{theorem}\label{thm:general}
Let $\ff \colon \RR \to D(L)$ be such that
\begin{align}
&\label{assum:1} \mbox{$t\mapsto \ff(t)$ is differentiable in $(Y,||\cdot||)$,}\\
&\label{assum:2} \mbox{$t\mapsto\ff'(t)$ is continuous in $(Y,||\cdot||)$,}\\
&\label{assum:3} \mbox{$t\mapsto L\ff(t)$ is continuous in $(Y,||\cdot||)$,}\\
&\label{assum:4} \mbox{$t\mapsto\ff(t)$ has compact support in $\RR$.}
\end{align}
Then
\begin{align}\label{eq:general}
\int_s^{\infty} T_{u-s} \Big[\ff'(u)+L\ff(u) \Big] du=-\ff(s)\,,\qquad s\in \RR\,,
\end{align}
where the integral is the Riemann 
integral of a Banach space valued function.
\end{theorem}
Theorem~\ref{thm:general} applies, e.g.,  to
$\ff(t)=f(t)\ff_0$ with $\ff_0\in D(L)$ and $f\in C_c^1(\RR)$.
Theorem \ref{thm:general} follows from  two auxiliary lemmas.
\begin{lemma}\label{lem:diff}
If $\ff$
satisfies \eqref{assum:1}, then $t\mapsto T_t \ff(t)$ is differentiable in $(Y,||\cdot||)$ and
\begin{align*}
\frac{d}{dt}\, T_t \ff(t)= T_t \ff'(t) + T_t L\ff(t)\,,\quad t\ge 0\,.
\end{align*}
\end{lemma}
For $t=0$ the derivative is understood as the right-hand derivative. The lemma is a version of the differentiation rule for products.
\begin{proof}[Proof of Lemma~\ref{lem:diff}]
Let $h\neq 0$ $(h>0$ if $t=0$) and $h \to 0$. Clearly,
\begin{align*}
&\frac{ T_{t+h} \ff(t+h) - T_t \ff(t)}{h} \\
&= T_{t+h} \ff'(t) +  T_{t+h}\left( \frac{\ff(t+h)-\ff(t)}{h}- \ff '(t) \right)+\left( \frac{T_{t+h}-T_t}{h} \right) \ff (t)\,.
\end{align*}
For some $M, \omega\ge 0$, we  have $||T_t||\le M e^{\omega t}$, $t\ge 0$ \cite{MR617913}. The lemma follows:
\begin{align*}
\left\| T_{t+h} \left( \frac{\ff (t+h)-\ff(t)}{h}- \ff '(t) \right) \right\| \le M e^{\omega (t+h)} \left\| \frac{\ff (t+h)-\ff (t)}{h}- \ff '(t) \right\|\to 0\,.
\end{align*}
\end{proof}

Let $a, b\in \RR$, $a<b$. We write $\ff \in C^1([a,b],Y)$ if $\ff\colon [a,b]\to Y$ and
\eqref{assum:1} and \eqref{assum:2} hold, with one-sided derivatives at the endpoints $a$ and $b$.
Here is the fundamental theorem of calculus for Riemann type Banach space integrals
(see \cite[Lemma 1.1.4]{MR838085} or \cite[Lemma 2.3.24]{MR1873235}).

\begin{lemma}\label{lem:Jakob}
If $\psi\in C^1([a,b],Y)$, then
$\int\limits_a^b\frac{d}{du}\left[ \psi(u) \right]du=\psi(b)-\psi(a)$.
\end{lemma}

\begin{proof}[Proof of Theorem~\ref{thm:general}]
Let
$s\in \RR$.
By Lemma \ref{lem:diff},  assumptions \eqref{assum:2}, \eqref{assum:3} and \eqref{assum:4}, and by  Lemma~\ref{lem:Jakob}, we obtain the result:
\begin{align}
\int_0^{\infty} T_{u} \Big[\ff'(u+s)+L\ff(u+s) \Big] du
&=\int_0^\infty \frac{d}{du} \left[
T_u \ff (u+s)
\right]du=
-\ff(s)\,.
\end{align}
In fact, if $s$ is fixed, the assumptions on $\ff(t)$ only need to hold in $[s,\infty)$.
\end{proof}

We shall give a partial converse to Theorem~\ref{thm:general} by showing that the infinitesimal generator of $T$ is the only operator $L$ that makes \eqref{eq:general} true.
\begin{theorem}\label{thm:unique1}
Let $A$ be a linear operator on a linear space $D(A)\subset Y$ with values in $Y$. Assume that $\ff \colon\RR \to D(A)$ is such that
\begin{align}
&\label{assum:1'} \mbox{$t\mapsto \ff(t)$ is differentiable in $(Y,||\cdot||)$,}\\
&\label{assum:2'} \mbox{$t\mapsto\ff'(t)$ is continuous in $(Y,||\cdot||)$,}\\
&\label{assum:3'} \mbox{$t\mapsto A\ff(t)$ is continuous in $(Y,||\cdot||)$,}\\
&\label{assum:4'} \mbox{$t\mapsto\ff(t)$ has compact support in $\RR$,}
\end{align}
\begin{align}\label{eq:unique1}
\int_s^{\infty} T_{u-s} \Big[\ff'(u)+A\ff(u) \Big] du=-\ff(s)\,,\qquad s\in \RR\,.
\end{align}
Then $\ff(t)\in D(L)$ and $L\ff(t) = A \ff(t)$ for all $t\in\RR$.
\end{theorem}
\begin{proof}
Let $t\in\RR$ and $h>0$. By \eqref{eq:unique1},
\begin{align*}
\int_{t+h}^{\infty} T_{u-t}  \Big[\ff'(u)+A\ff(u) \Big] du&=\int_{t+h}^{\infty} T_{u-(t+h)}T_{h}  \Big[\ff'(u)+A\ff(u) \Big] du\\
&=-T_h\ff(t+h)\,.
\end{align*}
Subtracting this from \eqref{eq:unique1} with $s=t$ we get
\begin{align*}
\int_t^{t+h} T_{u-t} \Big[\ff'(u)+A\ff(u) \Big] du=T_h\ff(t+h)-\ff(t)\,.
\end{align*}
We get
\begin{align*}
\left(\frac{T_h-I}{h} \right)\ff(t) = \frac{1}{h} \int_t^{t+h} T_{u-t} \Big[\ff'(u)+A\ff(u) \Big] du - T_h\left(\frac{\ff(t+h)-\ff(t)}{h} \right)\,.
\end{align*}
By \eqref{assum:1'}--\eqref{assum:3'} the limit on the right hand side exists as $h\to 0^+$ and equals
\begin{align*}
L\ff(t)= T_0\left( \ff'(t) + A\ff(t) \right) - T_0 \ff'(t)=A\ff(t)\,.
\end{align*}
In fact, the assumptions \eqref{assum:1'}--\eqref{eq:unique1} only need to hold
on $[t,t+\varepsilon)$, $\varepsilon>0$.
\end{proof}

\begin{remark}\label{rem:ef0}
We call $\ff$ satisfying \eqref{assum:1'}--\eqref{assum:4'} a {\it path for $A$}.
Define
$$D(A,T)=\{\ff(t)\colon \mbox{such that }  t\in \RR,\,\mbox{and}\, \ff \mbox{ is a path for } A \mbox{ satisfying } \eqref{eq:unique1}\}.$$
If $A$ is the infinitesimal generator of a strongly continuous semigroup $S=(S_t)_{t\geq 0}$ on $Y$ and $D(A,T)$ contains the cores of $L$ and $A$, then $L \equiv A$ and $T\equiv S$. Indeed, by the comment following Theorem \ref{thm:general}, for the infinitesimal generator $L$ of $T=(T_t)_{t\geq 0}$ we have $D(L,T)=D(L)$.
Theorem \ref{thm:unique1} means that $D(A,T)\subseteq D(A)\cap D(L)$,
and $A=L$ on $D(A,T)$. This identifies $L$ with $A$ and $T$ with $S$.
\end{remark}

We now focus on
L\'evy semigroups
discussed in the Introduction.
\begin{proof}[Proof of \eqref{eq:fsLp}]
Recall that $C_c^{\infty}(\RR^d)\subset C_0^2(\RR^d)\subset D(L)$.
We shall verify the assumptions of Theorem \ref{thm:general} for $\ff(t)= \phi(t,\cdot)$. It suffices to justify \eqref{assum:3}.
Recall that \eqref{formula:gen_Levy} holds for $f\in C_0^2(\RR^d)$ and $L$ is continuous from $C_0^2(\RR^d)$ to $C_0(\RR^d)$ \cite[p.~211]{MR1739520}.
We note that $t\mapsto \phi(t,\cdot)$ is continuous in $C_0^2(\RR^d)$.
Therefore
$t\mapsto L\phi(t,\cdot)$ is continuous in $(C_0(\RR^d), ||\cdot ||_{\infty})$.
By Theorem \ref{thm:general},
\begin{align*}
-\ff(s) &= \int_s^{\infty}\!\! P_{u-s} \Big[\ff'(u)+L\ff(u) \Big] du
\end{align*}
in $C_0(\RR^d)$.
Recall that
the Riemann
integrals
converge in norm.
Evaluation at a point is continuous on $(C_0(\RR^d), ||\cdot ||_{\infty})$, therefore the above identity holds pointwise, i.e. \eqref{eq:fsLp} holds. We note in passing that the integral in \eqref{eq:fsLp} may be interpreted as 
Lebesgue integral on $\RR\times \Rd$.
\end{proof}
\begin{theorem}[Uniqueness]
Let $C_c^{\infty}(\RR^d)$ be a core of a closed linear operator $A$ with domain $D(A)\subset (C_0(\RR^d), ||\cdot ||_{\infty})$.
If for all $s\in\RR$, $x\in\RR^d$ and $\phi\in C_c^{\infty}(\RR\times\RR^d)$,
\begin{align}\label{eq:unique2}
\int_s^{\infty} \int_{\RR^d} \Big[\partial_u \phi(u,x+z) + A\phi(u,x+z) \Big]p_{u-s}(dz)du = -\phi(s,x)\,,
\end{align}
then $A\equiv L$.
\end{theorem}
\begin{proof}
For $\varphi\in C_c^{\infty}(\Rd)$ and $f\in C_c^1(\Rd)$ we let $\ff(t)=f(t)\varphi$. Then $\ff$ is a path for $A$ and 
$\zeta(t):=\int_t^{\infty}\!\! P_{u-t} \big[\ff'(u)+A\ff(u) \big]du\in C_0(\Rd)$ converges in norm. By continuity of evaluations  and
\eqref{eq:unique2} with $\phi(t,x)=f(t)\varphi(x)$ we have $\zeta(t)(x)=-\ff(t)(x)$, $t\in\RR$, $x\in\Rd$. By Theorem \ref{thm:unique1}, $A=L$ on the common core $C_c^{\infty}(\Rd)$. This ends the proof.
\end{proof}

\begin{remark}
\rm
If the L\'evy process $\{X_t\}$
has a (transition) density function, i.e. $p_t(dy)=p(t,y)dy$
for $t> 0$,
then \eqref{eq:fsLp} reads as
\begin{align*}
\int_s^{\infty} \int_{\RR^d} p(u-s,z-x)\Big[\partial_u \phi(u,z) + L\phi(u,z) \Big]dzdu = -\phi(s,x).
\end{align*}
\end{remark}

We shall focus on the case when $d=1$ and  $\{X_t\}$ is a subordinator, i.e., a nondecreasing L\'evy process.
The L\'evy measure $\nu$ of $X_t$ is concentrated on $(0,\infty)$.
Since $\int (x\wedge 1)\nu(dx)<\infty$ and $L$ is a closed operator,
\eqref{formula:gen_Levy} may be rearranged: we obtain
$C_0^1(\RR)\subset D(L)$ and
\begin{align*}
Lf(x) = b\, \frac{d f}{dx}(x) + \int_0^{\infty} \Big( f(x+y)-f(x) \Big)\nu(dy)\,,\qquad f\in C_0^1(\RR)\,.
\end{align*}
Here $b\geq 0$ is the drift coefficient. Furthermore, for $f\in C_c^1(\RR)$ we obtain
\begin{align*}
\int_0^{\infty} \Big( f(x+y)-f(x) \Big)\nu(dy)&= \int_0^{\infty} \int_0^y f'(x+z)\,dz\,\nu(dy)\\
&= \int_0^{\infty} f'(x+z) \left(\int_z^{\infty} \nu(dy)\right)dz\,.
\end{align*}
Let $\overline{\nu}(z)=\int_z^{\infty}\nu(dy)$.
We thus have
\begin{align}\label{eq:gen_sub}
Lf(x)= b\,
f'(x) + \int_x^{\infty} f'(z)\, \overline{\nu}(z-x)\,dz\,,\qquad f\in C_c^1(\RR)\,.
\end{align}

\begin{example}
Let $\alpha\in(0,1)$ and $\{X_t\}$ be the $\alpha$-stable subordinator, i.e.,
\begin{align*}
b=0\quad \mbox{and}\quad \nu(dy)= \frac{\alpha}{\Gamma(1-\alpha)} \,y^{-\alpha-1} 1_{y>0}\,dy \,.
\end{align*}
We then see that the generator of $\{X_t\}$ coincides on $C_c^1(\RR)$ with the Weyl
fractional derivative (cf. \eqref{eq:gen_sub_12} for the case $\alpha=1/2$).
The potential operator for $\{X_t\}$ is
the Weyl fractional integral
\begin{align*}
W^{-\alpha}f(x) = \int_0^{\infty} T_t f(x)\,dt = \frac{1}{\Gamma(\alpha)} \int_x^{\infty} f(z)(z-x)^{\alpha-1}\,dz\,,\quad f\in C_c(\RR)\,.
\end{align*}
\end{example}
We note in passing that $-W^{-\alpha} \partial^{\alpha} = I$ (the identity operator) on $C_c^1(\RR)$.
Schr{\"o}dinger perturbations of $W^{-\alpha}$
were discussed in \cite[Example 2 and 3]{MR3000465}. The discussion was facilitated by the fact that the 3G Theorem holds for $(y-x)_+^{\alpha-1}/\Gamma(\alpha)$.

\begin{example}\label{ex:inv_Gauss_gen}
Since 
the inverse Gaussian subordinator is obtained by the Esscher transform  (tempering) 
and time rescaling
of the $1/2$-stable subordinator (cf. \cite{MR2042661}, Sec.~4.4.2),  for $f\in C_c^1(\RR)$ the generator of the inverse Gaussian subordinator is  given by \eqref{eq:giG}.
\end{example}

\begin{remark}\label{r:sp}
\rm
For (signed) $q\colon \RR\times X \to \RR$
we define the Schr\"odinger perturbation $\tp$ of $p$ by $q$ by exactly the same formulas \eqref{def:tp} and \eqref{def:p_n}. We get \eqref{eq:pfv}, \eqref{eq:pf}, Chapman-Kolmogorov,  provided the perturbation series for $|q|$, which gives an upper bound for $\tp$, is finite. Under this condition Lemma~\ref{cor:fst} remains valid, too. For lower bounds of $\tp$ for signed $q$ we refer to \cite{MR2457489,MR3200161}
\end{remark}

\subsection*{Acknowledgment}
We thank Tomasz Jakubowski and Sebastian Sydor for discussions and 
Zhen-Qing Chen for a question motivating Theorem~\ref{thm:unique1}.
We thank the referees for very helpful suggestions, which considerably improved the paper.
Krzysztof Bogdan was partially supported by NCN grant 2012/07/B/ST1/03356. Karol Szczypkowski was partially supported by NCN grant  2011/03/N/ST1/00607.

\end{document}